\documentclass[11pt,reqno]{amsart}

\usepackage{graphicx}

\usepackage{xspace}
\usepackage{amsmath,amsthm,amsfonts,amssymb,latexsym,mathrsfs,color,extarrows}
\usepackage{hyperref}

\newtheorem{theorem}{Theorem}
\newtheorem{corollary}[theorem]{Corollary}
\newtheorem{proposition}[theorem]{Proposition}

\newtheorem{lemma}[theorem]{Lemma}
\newtheorem{definition}[theorem]{Definition}
\newtheorem{example}[theorem]{Example}

\newcommand{\plat}{{\rm plat\,}}

\newcommand{\des}{{\rm des\,}}

\newcommand{\cc}{{\mathcal C}}

\newcommand{\D}{\mathfrak{D}}

\newcommand{\msn}{\mathfrak{S}_n}

\newcommand{\lrf}[1]{\lfloor #1\rfloor}

\newcommand{\mpp}{{\mathcal P}}

\newcommand{\z}{ \mathbb{Z}}
\newcommand{\asc}{{\rm asc\,}}

\DeclareMathOperator{\R}{\mathbb{R}}

\newcommand{\Stirling}[2]{\genfrac{\{}{\}}{0pt}{}{#1}{#2}}

\title{Eulerian polynomials for multipermutations and signed multipermutations}
\author[S.-M.~Ma]{Shi-Mei Ma}
\address{School of Mathematics and Statistics,
        Northeastern University at Qinhuangdao,
         Hebei 066000, P.R. China}
\email{shimeimapapers@163.com (S.-M. Ma)}
\author[J.~Ma]{Jun Ma}
\address{Department of mathematics, Shanghai Jiao Tong University, Shanghai, P.R. China}
\email{majun904@sjtu.edu.cn(J.~Ma)}
\author[Y.-N. Yeh]{Yeong-Nan Yeh}
\address{Institute of Mathematics,
        Academia Sinica, Taipei, Taiwan}
\email{mayeh@math.sinica.edu.tw (Y.-N. Yeh)}
\begin{document}

\maketitle
\begin{abstract}
In this paper, we study Eulerian polynomials for permutations and signed permutations of the multiset $\{1,1,2,2,\ldots,n,n\}$. Properties of these polynomials, including recurrence relations and unimodality are discussed. In particular, we give a unified proof of the fact that these polynomials
are unimodal with modes in the middle.
\bigskip

\end{abstract}
\date{\today}
\section{Introduction}
Let $[n]=\{1,2,\ldots,n\}$.
Let $\msn$ denote the symmetric group of all permutations of $[n]$ and let $\pi=\pi_1\cdots\pi_n\in\msn$.
A {\it descent} of $\pi$ is an index $i\in[n-1]$ such that $\pi_i>\pi_{i+1}$. Let $\des(\pi)$ denote the number of descents of $\pi$.
The {\it type $A$ Eulerian polynomial} is defined by
$$A_n(x)=\sum_{\pi\in\msn}x^{\des(\pi)}.$$
Let $\pm[n]=[n]\cup\{-1,-2,\ldots,-n\}$.
Denote by $B_n$ the hyperoctahedral group of rank $n$.
Elements of $B_n$ are signed permutations of $\pm[n]$ with the property that $\pi(-i)=-\pi(i)$ for all $i\in [n]$.
The {\it type $B$ Eulerian polynomial} is defined by
$$B_n(x)=\sum_{\pi\in B_n}x^{\des_B(\pi)},$$
where
$\des_B(\pi)=\#\{i\in\{0,1,2,\ldots,n-1\}: ~\pi(i)>\pi({i+1})\}$ and $\pi(0)=0$ (see~\cite{Brenti94}).

Let $f(x)=\sum_{i=0}^nf_ix^i$ be a polynomial with nonnegative coefficients.
We say that $f(x)$ is {\it unimodal} if there exists an index $m$ such that $f_0\leq f_1\leq \cdots \leq f_{m-1}\leq f_m\geq f_{m+1}\geq\cdots \geq f_n$.
Such an index $m$ is called a {\it mode} of $f(x)$.
If $f(x)$ is symmetric, then it can be expanded as
$$f(x)=\sum_{k=0}^{\lrf{{n}/{2}}}\gamma_kx^k(1+x)^{n-2k},$$ and it is said to be {\it $\gamma$-positive}
if $\gamma_k\geq 0$ for $0\leq k\leq \lrf{\frac{n}{2}}$.
The $\gamma$-positivity of $f(x)$ implies that $f(x)$ is unimodal and symmetric.
It is well known that Eulerian polynomials of types $A$ and $B$ are both $\gamma$-positive and so $A_n(x)$ and $B_n(x)$ are both unimodal with modes in the middle, see~\cite{Athanasiadis17,Ma19} and references therein.
In the past decades, the $\gamma$-positivity of enumerative polynomials have been extensively studied in combinatorial and geometric contexts (see, e.g.,~\cite{Gal05,Zeng12}).

Let $\rm s=\{s_i\}_{i\geq1}$ be a sequence of positive integers.
A geometric interpretation of Eulerian polynomials is obtained by considering the
$\rm{s}$-lecture hall polytope $\mpp_n^{(\rm{s})}$,
which is defined by
$$\mpp_n^{(\rm{s})}=\left\{(\lambda_1,\lambda_2,\ldots,\lambda_n)\in\mathbb{R}^n~\big{|}~ 0\leq\frac{\lambda_1}{s_1}\leq\frac{\lambda_2}{s_2}\leq\cdots \leq\frac{\lambda_n}{s_n}\leq 1\right\}.$$
Set $e_0=0$ and $s_0=1$.
Let $\rm{I}_n^{(\rm s)}=\{e=(e_1,\ldots,e_n)\in\z^n \mid 0\leq e_i<s_i ~\text{for $i\in[n]$}\}$ be
the set of $n$-dimensional {\it$\rm s$-inversion sequences}.
The {\it $\rm s$-Eulerian polynomial} is defined as
$$E_n^{(\rm s)}(x)=\sum_{e\in \rm{I}_n^{(\rm s)}}x^{\asc(e)},$$ where
$\asc(e)=\#\left\{i\in\{0,1,2,\ldots,n-1\}\big{|}\frac{e_i}{s_i}<\frac{e_{i+1}}{s_{i+1}}\right\}$.
In particular, we have
$$E_n^{(1,2,\ldots,n)}(x)=A_n(x).$$
It is well known~\cite{Savage1202} that the $\rm{s}$-Eulerian polynomial is the $h^*$-polynomial of $\mpp_n^{(\rm{s})}$.
Savage and Visontai~\cite{Savage15} showed that for any sequence $\rm{s}$ of positive integers,
the $\rm s$-Eulerian polynomial $E_n^{(\rm s)}(x)$ has only real zeros.
Since unimodality is implied by real-rootedness, thus $E_n^{(\rm s)}(x)$ is unimodal for any sequence $\rm{s}$ of positive integers.
As pointed out by Savage and Visontai~\cite[Remark 4.2]{Savage15}, $E_n^{(\rm s)}(x)$ need not be symmetric. Thus a natural problem to determine the modes of $E_n^{(\rm s)}(x)$ for some $\rm s$.
In this paper, we shall determine the modes of Eulerian polynomials for multipermutations and signed multipermutations.

For any $n$-element multiset $M$, a {\it descent} (resp. {\it ascent}, {\it plateau}) in a multiset permutation $\sigma$ of $M$
is an index $i$ such that $\sigma_i>\sigma_{i+1}$ (resp. $\sigma_i<\sigma_{i+1}$, $\sigma_i=\sigma_{i+1}$), where $i\in[n-1]$.
Let $\des(\sigma)$ (resp.~$\asc(\sigma)$, $\plat(\sigma)$) denote the number of descents (resp. ascents, plateaus) of $\sigma$, i.e.,
\begin{align*}
\des(\sigma)&=\{i\in [n-1]: \sigma_i>\sigma_{i+1}\},\\
\asc(\sigma)&=\{i\in [n-1]: \sigma_i<\sigma_{i+1}\},\\
\plat(\sigma)&=\{i\in [n-1]: \sigma_i=\sigma_{i+1}\}.
\end{align*}
Let $[n]_2=\{1,1,2,2,\ldots,n,n\}$.
Denote by $\cc_n$ the set of permutations of $[n]_2$. Let $$[\overline{n}]_2=[n]_2 \cup\{n+1\}=\{1,1,2,2,\ldots,n,n,n+1\}.$$
Let $\D_n$ denote the set of permutations of $[\overline{n}]_2$.
We define
\begin{align*}
P_n(x)&=\sum_{\sigma \in\cc_n}x^{\des(\sigma)}=\sum_{k=0}^{2n-2}P_{n,k}x^k,\\
Q_n(x)&=\sum_{\sigma \in \D_n}x^{\des(\sigma)}=\sum_{k=0}^{2n-1}Q_{n,k}x^k.
\end{align*}
In particular,
$P_1(x)=1,Q_1(x)=1+2x,P_2(x)=1+4x+x^2,Q_2(x)=1+12x+15x^2+2x^3$.
Let $\rm{s}=(1,1,3,2,5,\ldots)$, where $s_{2i}=i,s_{2i-1}=2i-1$ for $i\geq 1$.
By using the identity $$\sum_{k\geq 0}\binom{k+2}{2}^nx^k=\frac{{P_n(x)}}{(1-x)^{2n+1}},$$
Savage and Visontai~\cite[Theorem~3.23]{Savage15} showed that
\begin{equation}\label{Pnxdef}
P_n(x)=E_{2n}^{(1,1,3,2,\ldots,2n-1,n)}(x).
\end{equation}
And so $P_n(x)$ has only real zeros for any $n\geq 1$.
According to~\cite[Remark 3.1.1]{Gal05}, we have the following result.
\begin{proposition}\label{gamma}
If $f(x)$ is symmetric and has only real negative zeros,
then $f(x)$ is $\gamma$-positive.
\end{proposition}

It should be noted that the polynomial $P_n(x)$ has been studied by Carlitz and Hoggatt~\cite{Carlitz78}.
They found that $P_n(x)$ is a symmetric polynomial. Then combining~\eqref{Pnxdef} and Proposition~\ref{gamma}, we get the following result.
\begin{proposition}\label{Pnxgamma}
For any $n\geq 1$, the polynomial $P_n(x)$ is $\gamma$-positive, and so it is unimodal with mode in the middle.
\end{proposition}

It follows from~\cite[Eq.~(2.8)]{Carlitz78} that
the numbers $P_{n,k}$ satisfy the recurrence relation
\begin{equation}\label{Pnk-recurrence}
P_{n+1,k}=\binom{k+2}{2}P_{n,k}+(k+1)(2n-k+1)P_{n,k-1}+\binom{2n-k+2}{2}P_{n,k-2}.
\end{equation}
with $P_{1,0}=1$ and $P_{1,k}=0$ for $k\neq0$.
From~\cite[Theorem~14]{Savage1202}, we see that
$$\sum_{t\geq 0}(t+1)^{n+1}\left(\frac{t+2}{2}\right)^{n}x^t=\frac{E_{2n+1}^{(1,1,3,2,5,3,7,\ldots,2n-1,n,2n+1)}(x)}{(1-x)^{2n+2}}.$$
A classical result of MacMahon~\cite[Vol~2, Chapter IV, p.~211]{MacMahon20} says that
$$\frac{\sum_{\pi\in P(\{1^{p_1},2^{p_2},\ldots,n^{p_n}\})}x^{\des(\pi)}}{(1-x)^{1+\sum_{i=1}^n}p_i}
=\sum_{t\geq 0}\frac{(t+1)\cdots(t+p_1)\cdots(t+1)\cdots(t+p_n)}{p_1!\cdots p_n!}x^t.$$
where $P(\{1^{p_1},2^{p_2},\ldots,n^{p_n}\})$ denote the set of multipermutations of $\{1^{p_1},\ldots,n^{p_n}\}$. Therefore,
when $p_1=\cdots=p_n=2$ and $p_{n+1}=1$, we immediately obtain
$$Q_n(x)=E_{2n+1}^{(1,1,3,2,\ldots,2n-1,n,2n+1)}(x).$$

The number of {\it descents} and {\it ascents} of a signed permutation $\sigma=\sigma_1\sigma_2\cdots \sigma_n$ are defined as
\begin{align*}
\des_B(\sigma)&=\{i\in\{0,1,2,\ldots,n-1\}: \sigma_i>\sigma_{i+1}\},\\
\asc_B(\sigma)&=\{i\in\{0,1,2,\ldots,n-1\}: \sigma_i<\sigma_{i+1}\}.
\end{align*}
where $\sigma_0=0$.
Let $\cc_n^{\pm}$ be the set of all signed permutations of the multiset $[n]_2$. Elements of $\cc_n^{\pm}$ are those of the form
$\pm\sigma_1\pm \sigma_2\cdots \pm \sigma_{2n}$, where $\sigma_1\cdots \sigma_{2n}\in\cc_n$.
We write $-i$ by
$\overline{i}$ for each $i\in [n]$.
Let $\rm{s}=(1,4,\ldots,2n-1,4n)$, where $s_{2i}=4i,s_{2i-1}=2i-1$ for $i\geq 1$.
The following result was first conjectured by Savage and Visontai~\cite[Conjecture 3.25]{Savage15}, and solved
independently by Chen et al.~\cite{Chen16} and Lin~\cite{Lin1501}:
$$\sum_{\sigma \in\cc^{\pm}_n}x^{\des_B(\sigma)}=E_{2n}^{(1,4,3,8,\ldots,2n-1,4n)}(x).$$

Let $\D^{\pm}_n$ be the subset of signed permutations of $[\overline{n}]_2$ consisting of signed permutations such that the
element $n+1$ with a positive sign.
Thus any element of $\D^{\pm}_n$ can be
generated from an element of $\cc_n^{\pm}$ by inserting the entry $n+1$.
For example,
$$\D^{\pm}_1=\{112,\overline{1}12,1\overline{1}2,\overline{1}~\overline{1}2,121,\overline{1}21,12\overline{1},\overline{1}2\overline{1},211,2\overline{1}1,
21\overline{1},2\overline{1}~\overline{1}\}.$$
Clearly, $\#\D^{\pm}_n=(2n+1)\#\cc_n^{\pm}$.
We define
\begin{align*}
S_n(x)&=\sum_{\sigma \in\cc^{\pm}_n}x^{\des_B(\sigma)}=\sum_{i=0}^{2n-1}S_{n,i}x^i,\\
T_n(x)&=\sum_{\sigma \in \D^{\pm}_n}x^{\des_B(\sigma)}=\sum_{k=0}^{2n}T_{n,k}x^k.
\end{align*}
In particular,
\begin{align*}
S_1(x)&=1+3x,\\
T_1(x)&=1+8x+3x^2,\\
S_2(x)&=1+31x+55x^2+9x^3,\\
T_2(x)&=1+66x+258x^2+146x^3+9x^4,\\
S_3(x)&=1+209x+1884x^2+2828x^3+811x^4+27x^5.
\end{align*}
Following~\cite[Lemma~4]{Lin1501}, the numbers $S_{n,i}$ satisfy the recurrence relation
\begin{equation}\label{Snk-recu}
S_{n+1,i}=\binom{2i+2}{2}S_{n,i}+(2i(4n-2i+3)+2n+1)S_{n,i-1}+\binom{4n-2i+5}{2}S_{n,i-2},
\end{equation}
with the initial conditions $S_{1,0}=1,S_{1,1}=3$ and $S_{1,i}=0$ for $i<0$ or $i>1$.

Let $V_n$ be subset of signed permutations of $[\overline{n}]_2$ consisting of signed permutations such that the
element $n+1$ carries a minus sign. For $\sigma\in V_n$,
let $\des_r(\sigma)=\{i\in [2n+1]\mid \sigma_i>\sigma_{i+1}\}$,
where $\sigma_{2n+2}=0$.
Chen et al.~\cite[Theorem~3.1]{Chen16} proved that
$$\sum_{\sigma \in\D^{\pm}_n}x^{\des_r(\sigma)}=E_{2n+1}^{(1,4,3,8,\ldots,2n-1,4n,2n+1)}(x).$$
Using the bijections
$$\sigma=\sigma_1\sigma_2\cdots\sigma_n\mapsto \sigma'=\sigma_n\sigma_{n-1}\cdots\sigma_1\mapsto (-\sigma_n)\cdots(-\sigma_1),$$
we immediately obtain
$$\sum_{\sigma \in\D^{\pm}_n}x^{\des_r(\sigma)}=\sum_{\sigma \in\D^{\pm}_n}x^{\des_B(\sigma)}.$$
Hence
$$\sum_{\sigma \in\D^{\pm}_n}x^{\des_B(\sigma)}=E_{2n+1}^{(1,4,3,8,\ldots,2n-1,4n,2n+1)}(x).$$

The main purpose of this paper is to prove the following result.
\begin{theorem}\label{unimodal}
For any $n\geq 1$, the polynomials $Q_n(x),S_n(x)$ and $T_n(x)$ are all unimodal with modes in the middle.
\end{theorem}

In the next section, we recall some definitions that will be used throughout the rest of this paper.
In Section~\ref{Section3}, we prove Theorem~\ref{unimodal}. Moreover, combining~\eqref{QnxRnx} and Lemma~\ref{lemma02} as well as~\eqref{snk} and~\eqref{tnk}, it is easy to verify the following result.
\begin{theorem}\label{recurrencesystem}
Set $P_0(x)=Q_0(x)=S_0(x)=T_0(x)=1$. Then for $n\geq 0$, we have
\begin{equation*}
\left\{
  \begin{array}{l}
   Q_n(x)=(1+2nx)P_n(x)+x(1-x)P_n'(x),\\
P_{n+1}(x)=(1+nx)Q_n(x)+\frac{1}{2}x(1-x)Q_n'(x),
  \end{array}
\right.
\end{equation*}
and
\begin{equation*}
\left\{
  \begin{array}{l}
 T_n(x)=(1+2nx)S_n(x)+x(1-x)S_n'(x),\\
S_{n+1}(x)=(1+3x+4nx)T_n(x)+2x(1-x)T_n'(x).
  \end{array}
\right.
\end{equation*}
\end{theorem}

Suppose that $p,q\in\R[x]$ both have only real zeros,
that those of $p$ are $\xi_1<\cdots<\xi_n$,
and that those of $q$ are $\theta_1<\cdots<\theta_m$.
We say that $p$ {\it strictly interlaces} $q$ if $\deg q=1+\deg p$ and the zeros of
$p$ and $q$ satisfy
\begin{equation*}\label{zeros}
\theta_1<\xi_1<\theta_2<\cdots<\xi_n
<\theta_{n+1}.
\end{equation*}
We use the notation $p\prec_{int} q$ for ``$p$ strictly interlaces $q$''.
Combining Theorem~\ref{recurrencesystem} and Theorem~\cite[Theorem~2.1]{Liu07}, we immediately get the following corollary.
\begin{corollary}
For $n\geq 1$, the polynomials $P_n(x),Q_n(x),T_n(x)$ and $S_n(x)$ have only real zeros. Moreover, we have
$P_n(x)\prec_{int} Q_n(x)\prec_{int} P_{n+1}(x)$ and $S_n(x)\prec_{int} T_n(x)\prec_{int} S_{n+1}(x)$.
\end{corollary}
\section{Preliminary}
Following~\cite[Definition~2.9]{Schepers13},
we say that $f(x)$ is {\it alternatingly increasing} if
$$f_0\leq f_n\leq f_1\leq f_{n-1}\leq\cdots f_{\lrf{\frac{n+1}{2}}}.$$
Clearly, alternatingly increasing property is a stronger property than unimodality.
We now recall an elementary result.
\begin{proposition}[{\cite{Beck2010,Branden18}}]\label{prop01}
Let $f(x)$ be a polynomial of degree $n$.
There is a unique symmetric decomposition $f(x)= a(x)+xb(x)$, where \begin{equation}\label{ax-bx-prop01}
a(x)=\frac{f(x)-x^{n+1}f(1/x)}{1-x},~b(x)=\frac{x^nf(1/x)-f(x)}{1-x}.
\end{equation}
\end{proposition}
From~\eqref{ax-bx-prop01}, we see that $\deg a(x)=n$ and $\deg b(x)\leq n-1$. In particular, if $f(x)$ is symmetric, then $a(x)=f(x)$ and $b(x)=0$.
Note that $a(x)$ and $b(x)$ are both symmetric.
We call the ordered pair of polynomials $(a(x),b(x))$ the {\it symmetric decomposition} of $f(x)$.
The symmetric decomposition of polynomials was recently studied in Ehrhart theory.
The reader is referred to~\cite[Remark~4.1]{Solus19}
for the combinatorial interpretation of the symmetric decomposition of $h^*$-polynomials of lattice polytopes.

As pointed out by Br\"and\'en and Solus~\cite{Branden18},
the polynomial $f(x)$ is alternatingly increasing if and only if the pair of polynomials in its symmetric decomposition are both unimodal
and have nonnegative coefficients.
Br\"and\'en and Solus~\cite{Branden18} related the alternatingly increasing property to real-rootedness of
the symmetric decomposition.
This paper presents the continuation of studies on alternatingly increasing property and we relate it to bi-$\gamma$-positivity.
\begin{definition}
Let $(a(x),b(x))$ be the symmetric decomposition of the polynomial $f(x)$. If $a(x)$ and $b(x)$ are both $\gamma$-positive, then we say that
$f(x)$ is {\it bi-$\gamma$-positive}.
\end{definition}
It is clear that if a polynomial is bi-$\gamma$-positive, then it is alternatingly increasing.
Moreover, any $\gamma$-positive polynomial is also bi-$\gamma$-positive. The reader is referred to~\cite{Athanasiadis14,Athanasiadis18} for some examples of bi-$\gamma$-positive polynomials.

The main technique used in this paper is the context-free grammar.
Let $V$ be an alphabet whose letters are regarded as independent commutative
indeterminates. Following Chen~\cite{Chen93},
a {\it context-free grammar} $G$ over $V$ is a set of substitution rules replacing a variable in $V$ by
a formal function of variables in $V$.
The formal derivative $D:=D_G$ with respect to $G$ is defined as a linear operator
such that each substitution rule is treated as the common differential rule. For two formal functions $u$ and $v$,
we have $D(u+v)=D(u)+D(v),~D(uv)=D(u)v+uD(v)$.
For a constant $c$, we have $D(c)=0$.
\begin{example}[{\cite{Chen93}}]
If $G=\{x\rightarrow xy, y\rightarrow y\}$. Then $D_G(x)=xy,D_G(y)=y,D_G^2(x)=x(y+y^2)$.
In general, we have
$D_G^n(x)=x\sum_{k=0}^n\Stirling{n}{k}y^k$,
where $\Stirling{n}{k}$ is the Stirling number of the second kind.
\end{example}

Context-free grammar is a powerful tool to study exponential structure, and can be used to deduce convolution formulas and exponential generating functions of enumerative polynomials (see~\cite{Chen17,Fu18}).
Very recently, by using the theory of context-free grammars, we studied $\gamma$-positivity of several descent-type polynomials~\cite{Ma19}.
%
\section{Proof of Theorem~\ref{unimodal}}\label{Section3}
\subsection{A grammatical labeling of signed multipermutations}
\hspace*{\parindent}%

Following~\cite{Chen17}, a {\it grammatical labeling} is an assignment of the underlying elements of a combinatorial structure
with variables, which is consistent with the substitution rules of a grammar.
In the following discussion, we assume that signed multipermutations are appended by 0.
Let $N(\sigma)$ be the number of negative entries of $\sigma$. Define
$$\des^*(\sigma)=\des_B(\sigma0),~\asc^*(\sigma)=\asc_B(\sigma0).$$

\begin{lemma}\label{lemma-signed}
Let $G_1=\{x\rightarrow w,~y\rightarrow w\}$
and
\begin{equation*}\label{G2-q}
G_2=\left\{x\rightarrow \frac{(1+q)^2x^2y^2}{2w},~y\rightarrow\frac{(1+q)^2x^2y^2}{2w},~w\rightarrow xy(q(x+y)+(1+q^2)y)\right\}.
\end{equation*}
Then we have
\begin{equation*}
(D_2D_1)^n(x)=\sum_{\sigma \in\cc_n^{\pm}}x^{\des^*(\sigma)}y^{\asc^*(\sigma)+\plat(\sigma)}q^{N(\sigma)},
\end{equation*}
\begin{equation*}
D_1(D_2D_1)^{n}(x)=w\sum_{\sigma \in\D_n^{\pm}}x^{\des^*(\sigma)-1}y^{\asc^*(\sigma)+\plat(\sigma)-1}q^{N(\sigma)}.
\end{equation*}
\end{lemma}
\begin{proof}
We now introduce a grammatical labeling of signed multipermutations.
We first put a subscript label $q$ right after each negative entry of a signed permutation.
For $\sigma\in\cc^{\pm}_n$, we label a descent of $\sigma$ by $x$ and label an ascent or a plateau by $y$.
Thus the weight of $\sigma\in\cc_n^{\pm}$ is given by $$W(\sigma)=x^{\des^*(\sigma)}y^{\asc^*(\sigma)+\plat(\sigma)}q^{N(\sigma)}.$$
Note that $\cc^{\pm}_1=\{^y1^y1^x,^y1^x\overline{1}_q^y,^x\overline{1}_q^y1^x,^x\overline{1}_q^y~\overline{1}^y_q\}$.
Clearly, the sum of weights of the elements in $\cc^{\pm}_1$ is given by $D_2D_1(x)$.

A grammatical labeling of $\sigma'\in\D_n^{\pm}$ is given as follows.
If $i$ is a descent and $\sigma'_i\neq n+1$, then put a label $x$ right after $\sigma'_i$.
If $i$ a plateau or an ascent and $\sigma'_{i+1}\neq n+1$, then put a label $y$ right after $\sigma'_i$.
If $\sigma'_i= n+1$, then put a label $w$ above $\sigma'_i$.
Hence the weight of $\sigma'\in\D^{\pm}_n$ is given by
$$W(\sigma')=wx^{\des^*(\sigma)-1}y^{\plat(\sigma)+\asc^*(\sigma)-1}q^{N(\sigma)}.$$
For example, ${\overset{w}2}\overline{1}_q^y1^x\in\D^{\pm}_1$.
It is easy to verify that
\begin{align*}
D_2D_1(x)&=xy(q(x+y)+(1+q^2)y,\\
D_1(D_2D_1(x))&=w(q(x+y)^2+(1+q)^2xy+y(1+q^2)(x+y)).
\end{align*}
Note that the sum of weights of the elements in $\D^{\pm}_1$ is given by $D_1(D_2D_1(x))$.
Then the result holds for $n=1$. We proceed by induction on $n$. Assume that the results hold for $n-1$, where $n\geq 2$.
Now we insert the entry $n$ into $\sigma\in\cc^{\pm}_{n-1}$.
Note that the insertion of $n$ corresponds to one substitution rule in $G_1$, since we always replace $x$ or $y$ by $w$.
By definition, we see that the action of $D_1$ on elements of $\cc^{\pm}_{n-1}$ generates all the elements in $\D^{\pm}_{n-1}$, and each element of $\D^{\pm}_{n-1}$ can be generated exactly once. Let $\sigma\in\D^{\pm}_{n-1}$, and
suppose that $\sigma_i=n$.
Now we insert $n$ or $\overline{n}$ into $\sigma$.
We distinguish three cases:
\begin{itemize}
  \item[\rm($c_1$)]We can insert $n$ or $\overline{n}$ right after $\sigma_i$, or we can first replace $\sigma_i$ by $\overline{n}$, and then
insert $n$ or $\overline{n}$ right after $\overline{n}$. In this case, the
insertion corresponds to applying the substitution rule $w\rightarrow xy(q(x+y)+(1+q^2)y)$;
 \item[\rm($c_2$)]Suppose that $j$ is a descent, where $|i-j|\geq2$.
Then we can insert $n$ or $\overline{n}$ right after $\sigma_j$, or we can first replace $\sigma_i$ by $\overline{n}$, and then
insert $n$ or $\overline{n}$ right after $\sigma_j$. It should be noted that we can get the same signed permutation if the first $n$ in the $(j+1)$th position of $\sigma$.
In this case, the
insertion corresponds to applying the substitution rule $x\rightarrow \frac{(1+q)^2x^2y^2}{2w}$;
\item[\rm($c_3$)]Suppose that $j$ is an ascent or a plateau, where $|i-j|\geq2$. Then we can insert $n$ or $\overline{n}$ right after $\sigma_j$, or we can first replace $\sigma_i$ by $\overline{n}$, and then insert $n$ or $\overline{n}$ right after $\sigma_j$. Similarly, we can get the same signed permutation if the first $n$ in the $(j+1)$th position of $\sigma$.
In this case, the
insertion corresponds to applying the substitution rule $y\rightarrow \frac{(1+q)^2x^2y^2}{2w}$.
\end{itemize}
By induction, it is routine to check that
the action of $D_2$ on elements of $\D^{\pm}_{n-1}$ generates all the elements in $\cc^{\pm}_{n}$.
This completes the proof.
\end{proof}

When $q=0$, lemma~\ref{lemma-signed} reduces to the following lemma.
\begin{lemma}\label{lemma02}
Let $G_3=\{x\rightarrow w,~y\rightarrow w\}$
and $G_4=\left\{x\rightarrow \frac{x^2y^2}{2w},~y\rightarrow \frac{x^2y^2}{2w},~w\rightarrow xy^2\right\}$.
Then we have
\begin{equation}\label{Pnk-grammar}
(D_4D_3)^n(x)=x\sum_{\sigma\in\cc_n}x^{\des(\sigma)}y^{2n-\des(\sigma)}=x\sum_{k=0}^{2n-2}P_{n,k}x^{k}y^{2n-k},
\end{equation}
\begin{equation}\label{Qnk-grammar}
D_3(D_4D_3)^n(x)=w\sum_{\sigma\in\D_{n}}x^{\des(\sigma)}y^{2n-\des(\sigma)}=w\sum_{k=0}^{2n-1}Q_{n,k}x^{k}y^{2n-k}.
\end{equation}
\end{lemma}

When $q=1$, lemma~\ref{lemma-signed} reduces to the following lemma.
\begin{lemma}\label{lemma03}
Let
$G_5=\{x\rightarrow w,~y\rightarrow w\}$
and $$G_6=\left\{x\rightarrow \frac{2x^2y^2}{w},~y\rightarrow \frac{2x^2y^2}{w},~w\rightarrow xy(x+3y)\right\}.$$
Then we have
\begin{equation}\label{recu-signed03}
(D_6D_5)^n(x)=\sum_{\sigma \in\cc_n^{\pm}}x^{\des^*(\sigma)}y^{\asc^*(\sigma)+\plat(\sigma)},
\end{equation}
\begin{equation}\label{recu-signed04}
D_5(D_6D_5)^{n}(x)=w\sum_{\sigma \in\D_n^{\pm}}x^{\des^*(\sigma)-1}y^{\asc^*(\sigma)+\plat(\sigma)-1}.
\end{equation}
\end{lemma}

\subsection{Eulerian polynomials for multipermutations}
\hspace*{\parindent}%

In the following, we first give another proof of Proposition~\ref{Pnxgamma}.

\begin{proof}[A direct proof of Proposition~\ref{Pnxgamma}]
Let $G_3$ and $G_4$ be the grammars given in Lemma~\ref{lemma02}.
Note that $D_3(x)=w,~D_4D_3(x)=xy^2,~D_3(D_4D_3(x))=w(xy+y^2+xy)$.
Assume that there are nonnegative integers $p_{n,k}$ such that
\begin{equation}\label{Pnk-grammar02}
(D_4D_3)^n(x)=xy^2\sum_{k=0}^{2n-2}p_{n,k}(xy)^{k}(x+y)^{2n-2-2k}.
\end{equation}
Then we have
\begin{align*}
&(D_4D_3)^{n+1}(x)\\
&=D_4\left(w\sum_{k}p_{n,k}\left((k+1)x^{k}y^{k+1}(x+y)^{2n-1-2k}+x^{k+1}y^{k+1}(x+y)^{2n-2-2k}\right)\right)+\\
&D_4\left(2w\sum_{k}p_{n,k}(2n-2-2k)x^{k+1}y^{k+2}(x+y)^{2n-3-2k}\right)\\
&=xy^2\sum_{k}p_{n,k}\left((k+1)x^{k}y^{k+1}(x+y)^{2n-1-2k}+x^{k+1}y^{k+1}(x+y)^{2n-2-2k}\right)+\\
&xy^2\sum_{k}p_{n,k}2(2n-2-2k)x^{k+1}y^{k+2}(x+y)^{2n-3-2k}+\\
&xy^2\sum_{k}p_{n,k}\left(\frac{(k+1)k}{2}x^ky^{k+1}(x+y)^{2n-1-2k}+\frac{(k+1)^2}{2}x^{k+1}y^k(x+y)^{2n-1-2k}\right)+\\
&xy^2\sum_{k}p_{n,k}(k+1)(2n-1-2k)x^{k+1}y^{k+1}(x+y)^{2n-2-2k}+\\
&xy^2\sum_{k}p_{n,k}\left(\frac{k+1}{2}x^{k+1}y^{k+1}(x+y)^{2n-2-2k}+\frac{k+1}{2}x^{k+2}y^k(x+y)^{2n-2-2k}\right)+\\
&xy^2\sum_{k}p_{n,k}(2n-2-2k)x^{k+2}y^{k+1}(x+y)^{2n-3-2k}+\\
&xy^2\sum_{k}p_{n,k}(2n-2-2k)(x+y)^{2n-3-2k}\left((k+1)x^{k+1}y^{k+2}+(k+2)x^{k+2}y^{k+1}\right)+\\
&xy^2\sum_{k}p_{n,k}2(2n-2-2k)(2n-3-2k)x^{k+2}y^{k+2}(x+y)^{2n-4-2k}.
\end{align*}
Let
\begin{align*}
I_1&=xy^2\sum_{k}p_{n,k}(x+y)^{2n-2-2k}(x^{k+1}y^{k+1}+(k+1)(2n-1-2k)x^{k+1}y^{k+1})+\\
&xy^2\sum_{k}p_{n,k}2(2n-2-2k)(2n-3-2k)x^{k+2}y^{k+2}(x+y)^{2n-4-2k},\\
I_2&=xy^2\sum_{k}p_{n,k}(2n-2-2k)(2x^{k+1}y^{k+2}(x+y)^{2n-3-2k}+x^{k+2}y^{k+1}(x+y)^{2n-3-2k})+\\
&xy^2\sum_{k}p_{n,k}(2n-2-2k)(x+y)^{2n-3-2k}((k+1)x^{k+1}y^{k+2}+(k+2)x^{k+2}y^{k+1}),\\
I_3&=xy^2\sum_{k}p_{n,k}(k+1)x^{k}y^{k+1}(x+y)^{2n-1-2k}+\\
&\frac{1}{2}xy^2\sum_{k}p_{n,k}({(k+1)^2}x^{k+1}y^k(x+y)^{2n-1-2k}+{(k+1)k}x^ky^{k+1}(x+y)^{2n-1-2k})+\\
&\frac{1}{2}xy^2\sum_{k}p_{n,k}({(k+1)}x^{k+2}y^k(x+y)^{2n-2-2k}+{(k+1)}x^{k+1}y^{k+1}(x+y)^{2n-2-2k}).
\end{align*}
Then $(D_4D_3)^{n+1}(x)=I_1+I_2+I_3$.
Combining like terms, we get that
\begin{align*}
I_1&=xy^2\sum_{k}p_{n,k}(1+(k+1)(2n-1-2k))x^{k+1}y^{k+1}(x+y)^{2n-2-2k}+\\
&xy^2\sum_{k}p_{n,k}2(2n-2-2k)(2n-3-2k)x^{k+2}y^{k+2}(x+y)^{2n-4-2k},\\
I_2
&=xy^2\sum_{k}p_{n,k}(2n-2-2k)(k+3)x^{k+1}y^{k+1}(x+y)^{2n-2-2k},\\
I_3
&=xy^2\sum_{k}p_{n,k}\binom{k+2}{2}x^{k}y^k(x+y)^{2n-2k}.
\end{align*}
Comparing the coefficients of $xy^2(xy)^{k}(x+y)^{2n-2k}$ and simplifying yields
\begin{equation*}\label{pnk-recu01}
p_{n+1,k}=\binom{k+2}{2}p_{n,k}+(k+1)(4n-4k+1)p_{n,k-1}+4\binom{2n-2k+2}{2}p_{n,k-2},
\end{equation*}
with $p_{1,0}=1$ and $p_{1,k}=0$ for $k\geq 1$. Thus the expansion~\eqref{Pnk-grammar02} holds for $n+1$.
Comparing~\eqref{Pnk-grammar} with~\eqref{Pnk-grammar02},
we obtain
$$P_n(x)=\sum_{k=0}^{n-1}p_{n,k}x^{k}(1+x)^{2n-2-2k},$$ and so $P_n(x)$ is $\gamma$-positive. This completes the proof.
\end{proof}

\begin{theorem}\label{multipermu01}
For any $n\geq 1$, the polynomial $Q_n(x)$ is bi-$\gamma$-positive and so it is unimodal with mode in the middle.
\end{theorem}
\begin{proof}
It follows from~\eqref{Pnk-grammar} that
$$D_3(D_4D_3)^n(x)=w\sum_{k=0}^{2n-1}P_{n,k}x^ky^{2n-k-1}((k+1)y+(2n-k)x).$$
Define
\begin{equation}\label{Rnkdef}
R_{n,k}=(k+1)P_{n,k}+(2n-k)P_{n,k-1}
\end{equation}
We obtain
\begin{align*}
D_3(D_4D_3)^n(x)&=w\sum_{k=0}^{2n-1}R_{n,k}x^ky^{2n-k}+wx\sum_{k=0}^{2n-1}P_{n,k}x^{k}y^{2n-k-1}.
\end{align*}
Let $R_n(x)=\sum_{k=0}^{2n-1}R_{n,k}x^k$. By using~\eqref{Rnkdef}, we get
$$R_n(x)=(1+(2n-1)x)P_n(x)+x(1-x)P_n'(x).$$
Since $P_n(x)$ is symmetric, it is routine to verify that $R_{n,k}=R_{n,2n-1-k}$.
It follows from~\eqref{Qnk-grammar} that
\begin{equation}\label{QnxRnx}
Q_n(x)=R_n(x)+xP_n(x).
\end{equation}
By using~\eqref{Pnk-grammar02}, we obtain
\begin{align*}
&D_3(D_4D_3)^n(x)=w\sum_{k}p_{n,k}x^{k+1}y^{k+1}(x+y)^{2n-2-2k}+\\
&w\sum_{k}p_{n,k}(x+y)^{2n-3-2k}\left((k+1)x^{k}y^{k+1}(x+y)^{2}+2(2n-2-2k)x^{k+1}y^{k+2}\right).
\end{align*}
Set $r_{n,k}=(k+1)p_{n,k}+4(n-k)p_{n,k-1}$. Then
$$R_n(x)=\sum_{k\geq0}r_{n,k}x^{k}(1+x)^{2n-1-2k}.$$
Therefore, the polynomial $R_n(x)$ is $\gamma$-positive, and so $Q_n(x)$ is bi-$\gamma$-positive.
\end{proof}

Note that $$xP_n(x)=\sum_{\substack{\sigma \in\D_n\\\sigma_{1}=n+1}}x^{\des(\sigma)}.$$
By using~\eqref{QnxRnx}, we immediately get that
$$R_n(x)=\sum_{\substack{\sigma \in\D_n\\\sigma_{1}<n+1}}x^{\des(\sigma)}.$$
\subsection{Eulerian polynomials for signed multipermutations}
\hspace*{\parindent}%

Recall that
\begin{align*}
S_{n,k}&=\#\{\sigma\in \cc^{\pm}_n\mid \des_B(\sigma)=k\},\\
T_{n,k}&=\#\{\sigma\in \D^{\pm}_n\mid \des_B(\sigma)=k\}.
\end{align*}
Motivated by~\eqref{QnxRnx}, it is natural to derive a connection between the numbers $T_{n,k}$ and $S_{n,k}$.
There are two ways such that we can get an element $\sigma\in\D^{\pm}_n$ with $k$ descents from an element $\sigma'\in\cc^{\pm}_n$ by inserting the entry $n+1$.
If $\des(\sigma')=k$, then we can insert $n+1$ at the end of $\sigma'$, or put the entry $n+1$ between two entries that form a descent.
This gives $k+1$ choices for the positions of $n+1$.
If $\des(\sigma')=k-1$, then we can insert the entry $n+1$ into the other $2n-(k-1)$ positions.
So we have
\begin{equation}\label{Tnk-recu}
T_{n,k}=(k+1)S_{n,k}+(2n-k+1)S_{n,k-1}.
\end{equation}
Equivalently,
$$T_n(x)=(1+2nx)S_n(x)+x(1-x)S_n'(x).$$

\begin{theorem}\label{Rnk}
For any $n\geq 1$, the polynomials $S_n(x)$ and $T_n(x)$ are both bi-$\gamma$-positive, and so they are unimodal with modes in the middle.
\end{theorem}
\begin{proof}
Let $G_5$ and $G_6$ be the grammars given in Lemma~\ref{lemma03}.
Note that
$$D_5(x)=w,~D_6D_5(x)=xy(x+3y),~D_5D_6D_5(x))=w(x^2+8xy+3y^2).$$
Recall that the numbers $S_{n,k}$ satisfy the recurrence relation~\eqref{Snk-recu}.
Suppose that
$$(D_6D_5)^n(x)=\sum_{k=0}^{2n-1}\widetilde{S}_{n,k}x^{2n-k}y^{k+1}.$$
Then
\begin{align*}
&D_6(D_5(D_6D_5)^n(x))
=\sum_k {\widetilde{S}}_{n,k}((2n-k)(4n-2k+1)x^{2n-k}y^{k+3}+\\
&\sum_k {\widetilde{S}}_{n,k}((2n-k)(4k+5)+3(k+1))x^{2n-k+1}y^{k+2}+\sum_k {\widetilde{S}}_{n,k}\binom{2k+2}{2}x^{2n-k+2}y^{k+1}.
\end{align*}
Comparing the coefficient of $x^{2n+2-k}y^{k+1}$, we see that $\widetilde{S}_{n,k}$ satisfy the same recurrence relation
and initial conditions as $S_{n,k}$, so they agree.
Therefore,
\begin{equation}\label{snk}
(D_6D_5)^n(x)=\sum_{k=0}^{2n-1}S_{n,k}x^{2n-k}y^{k+1}.
\end{equation}
Note that
$$D_5(D_6D_5)^n(x)=w\sum_kS_{n,k}\left((2n-k)x^{2n-k-1}y^{k+1}+(k+1)x^{2n-k}y^k\right).$$
Then combining~\eqref{Tnk-recu}, we get
\begin{equation}\label{tnk}
D_5(D_6D_5)^n(x)=w\sum_{k=0}^{2n}T_{n,k}x^{2n-k}y^k.
\end{equation}

In the rest of the proof, we shall show the bi-$\gamma$-positivity of $S_n(x)$ and $T_n(x)$.
Note that
\begin{align*}
D_6D_5(x)&=xy(x+y)+2xy^2,\\
D_5(D_6D_5)(x)&=w((x+y)^2+4xy)+2wy(x+y).
\end{align*}
For $m\geq 1$, assume that the following expansions hold for $n=m$:
\begin{equation}\label{recu-signed0001}
\left\{
  \begin{array}{ll}
&(D_6D_5)^n(x)=xy\sum_{k}(xy)^k(x+y)^{2n-2-2k}(\eta^+_{2n,k}(x+y)+y\eta^-_{2n,k}),\\
&D_5(D_6D_5)^n(x)=w\sum_{k}(xy)^k(x+y)^{2n-1-2k}(\eta^+_{2n+1,k}(x+y)+y\eta^-_{2n+1,k}).
  \end{array}
\right.
\end{equation}
We proceed by induction.
Note that
\begin{align*}
&D_5(D_6D_5)^{m}
=w\sum_{k}\eta^+_{2m,k}(xy)^k(x+y)^{2m-2-2k}((k+1)(x+y)^{2}+2(2m-1-2k)xy)+\\
&w\sum_{k}\eta^-_{2m,k}y(xy)^k(x+y)^{2m-3-2k}((k+1)(x+y)^{2}+4(m-1-k)xy)+\\
&w\sum_{k}\eta^-_{2m,k}(xy)^{k+1}(x+y)^{2m-2-2k}.
\end{align*}
Comparing the coefficients of $(xy)^k(x+y)^{2m-2k}$ and $y(xy)^k(x+y)^{2m-1-2k}$, we obtain
\begin{align*}
\eta^+_{2m+1,k}&=(1+k)\eta^+_{2m,k}+2(2m-2k+1)\eta^+_{2m,k-1}+\eta^-_{2m,k-1},\\
\eta^-_{2m+1,k}&=(1+k)\eta^-_{2m,k}+4(m-k)\eta^-_{2m,k-1}.
\end{align*}
with the initial conditions $\eta^+_{1,0}=1,\eta^-_{1,0}=0,\eta^+_{2,0}=1$ and $\eta^-_{2,0}=2$.
Note that
\begin{align*}
&D_6(D_5(D_6D_5)^{m}(x))
=xy\sum_{k}\eta^+_{2m+1,k}(1+2k)x^ky^k(x+y)^{2m+1-2k}+\\
&2xy^2\sum_{k}\eta^+_{2m+1,k}x^ky^k(x+y)^{2m-2k}+\\
&8xy\sum_{k}\eta^+_{2m+1,k}(m-k)x^{k+1}y^{k+1}(x+y)^{2m-1-2k}+\\
&xy^2\sum_{k}\eta^-_{2m+1,k}(3+2k)x^ky^k(x+y)^{2m-2k}+\\
&4xy^2\sum_{k}\eta^-_{2m+1,k}(2m-1-2k)x^{k+1}y^{k+1}(x+y)^{2m-2-2k}.
\end{align*}
Comparing the coefficients of $xy(xy)^k(x+y)^{2m+1-2k}$ and $xy^2(xy)^k(x+y)^{2m-2k}$ in the above expression, we obtain
\begin{align*}
\eta^+_{2m+2,k}&=(1+2k)\eta^+_{2m+1,k}+8(m-k+1)\eta^+_{2m+1,k-1},\\
\eta^-_{2m+2,k}&=(3+2k)\eta^-_{2m+1,k}+4(2m-2k+1)\eta^-_{2m+1,k-1}+2\eta^+_{2m+1,k}.
\end{align*}
%
%
%
%
Thus~\eqref{recu-signed0001} holds for $n=m+1$.
We define
\begin{align*}
S_n^{+}(x)&=\sum_{k}\eta^+_{2n,k}x^k(1+x)^{2n-1-2k},~
S_n^{-}(x)=\sum_{k}\eta^-_{2n,k}x^k(1+x)^{2n-2-2k},\\
T_n^{+}(x)&=\sum_{k}\eta^+_{2n+1,k}x^k(1+x)^{2n-2k},~
T_n^{-}(x)=\sum_{k}\eta^-_{2n+1,k}(xy)^k(1+x)^{2n-1-2k}.
\end{align*}
Combining~\eqref{snk},~\eqref{tnk} and~\eqref{recu-signed0001}, we immediately get that $$S_n(x)=S_n^{+}(x)+xS_n^{-}(x),$$
 $$T_n(x)=T_n^{+}(x)+xT_n^{-}(x).$$
%
%
%
Therefore, the polynomials $S_n(x)$ and $T_n(x)$ are both bi-$\gamma$-positive.
\end{proof}

Combining~\eqref{snk},~\eqref{tnk} and Lemma~\ref{lemma03}, we see that
$$\#\{\sigma\in\cc^{\pm}_n\mid \asc^*(\sigma)+\plat(\sigma)=k+1\}=\#\{\sigma\in\cc^{\pm}_n\mid \des(\sigma)=k\},$$
$$\#\{\sigma\in\D^{\pm}_n\mid \asc^*(\sigma)+\plat(\sigma)=k+1\}=\#\{\sigma\in\D^{\pm}_n\mid \des(\sigma)=k\}.$$
Hence $$S_n(x)=\sum_{\sigma\in\cc^{\pm}_n}x^{\asc^*(\sigma)+\plat(\sigma)-1},~T_n(x)=\sum_{\sigma\in\D^{\pm}_n}x^{\asc^*(\sigma)+\plat(\sigma)-1}.$$
Equivalently, $x^{2n}S_n\left({1}/{x}\right)=\sum_{\sigma\in\cc_n^{\pm}}x^{\des^*(\sigma)}$ and $x^{2n+1}T_n\left({1}/{x}\right)=\sum_{\sigma\in\D_n^{\pm}}x^{\des^*(\sigma)}$.
Hence $$S^+_n(x)=\frac{\sum_{\sigma\in\cc_n^{\pm}}x^{\des(\sigma)}-\sum_{\sigma\in\cc_n^{\pm}}x^{\des^*(\sigma)}}{1-x}=\sum_{\substack{\sigma\in\cc_n^{\pm}\\ \sigma_{2n}>0\\}}x^{\des(\sigma)},$$
$$T^+_n(x)=\frac{\sum_{\sigma\in\D_n^{\pm}}x^{\des(\sigma)}-\sum_{\sigma\in\D_n^{\pm}}x^{\des^*(\sigma)}}{1-x}=\sum_{\substack{\sigma\in\D_n^{\pm}\\ \sigma_{2n}>0\\}}x^{\des(\sigma)}.$$

Therefore, we get the following result.
\begin{corollary}
For $n\geq 1$, we have
$$S^+_n(x)=\sum_{\substack{\sigma\in\cc_n^{\pm}\\ \sigma_{2n}>0\\}}x^{\des(\sigma)},~S^-_n(x)=\sum_{\substack{\sigma\in\cc_n^{\pm}\\ \sigma_{2n}<0\\}}x^{\des(\sigma)-1}.$$
$$T^+_n(x)=\sum_{\substack{\sigma\in\D_n^{\pm}\\ \sigma_{2n}>0\\}}x^{\des(\sigma)},~T^-_n(x)=\sum_{\substack{\sigma\in\D_n^{\pm}\\ \sigma_{2n}<0\\}}x^{\des(\sigma)-1}.$$
\end{corollary}


\end{document}